\documentclass[final,1p,times]{elsarticle}
\usepackage{amssymb}
\usepackage{amsmath}
\usepackage{amsthm}
\usepackage[utf8]{inputenc}
\usepackage{marginnote}
\usepackage{amsfonts}
\usepackage{textcomp}
\usepackage{color,soul}
\usepackage{placeins}
\usepackage{tikz}
\usepackage{relsize}
\usepackage{makecell}
\usepackage{amssymb}
\usepackage{graphicx}
\usepackage{hyperref}
\usepackage[none]{hyphenat}
\usepackage{lipsum}
\usepackage{tikz}
\newtheorem{thm}{Theorem}[section]
\newtheorem{lem}[thm]{Lemma}
\newtheorem{rem}[thm]{Remark}
\newtheorem{prop}[thm]{Proposition}
\newtheorem{cor}[thm]{Corollary}

\newtheorem{example}[thm]{Example}
\newdefinition{defn}[thm]{Definition}
\allowdisplaybreaks

\begin{document}
	\begin{frontmatter}
	
		\title{ Optimal Dual Frame Pairs: A Synergy with Graph Theory}
		\author{Shankhadeep Mondal}
		\ead{shankhadeep.mondal@ucf.edu}
		\author{Ram Narayan Mohapatra}
		\ead{Ram.Mohapatra@ucf.edu}

		\address{School of Mathematics, University of Central Florida, Orlando, Florida-32816}

		\begin{abstract}
			This paper investigates the optimization of dual frame pairs in the context of erasure problems in data transmission, using a graph theoretical approach. Frames are essential for mitigating errors and signal loss due to their redundancy properties. We address the use of spectral radius and operator norm for error measurements, presenting conditions for the optimality of dual pairs for one and two erasures. Our study shows that a tight frame generated by connected graphs and its canonical dual pair is optimal for one-erasure scenarios. Additionally, we compute the spectral radius of the error operator for one and two erasures in graph-generated frames, establishing necessary conditions for dual pair optimality.
		\end{abstract}

		
	\end{frontmatter}
	
	\section{Introduction}
	
	The utilization of frames in data transmission has gained prominence due to their inherent redundancy properties, which play a critical role in ensuring the integrity and accuracy of transmitted data. During the transmission process, various adverse global network conditions, such as congestion and limitations of the transmission channel, can cause errors. These errors may lead to the loss of certain parts of the encoded coefficients of a signal, resulting in what is known as erasures. The presence of redundancy within frames helps to counteract these erasures and enhances the accuracy of the reconstructed signal.
	
	In recent years, significant research efforts have been dedicated to addressing the problem of erasures by applying frame theory. One notable study by Casazza and Kovačević \cite{casa2} has delved into equal norm tight frames, exploring their properties, construction, and robustness to erasures in detail. Their work has provided a comprehensive understanding of how these frames can withstand erasures and maintain signal integrity. Another important contribution is from Goyal, Kovačević, and Kelner \cite{goya}, who examined uniform tight frames from the perspective of coding theory. They identified that these frames are optimal for one erasure, providing a solid foundation for further exploration in this area.
	
	Optimal dual frames are particularly crucial in minimizing the maximum error introduced by erasures. Holmes and Paulsen \cite{holm} introduced the concept of using the operator norm to measure the optimality of frames in the context of erasures. They focused on Parseval frames $F$ and their canonical dual $F$, analyzing the necessary and sufficient conditions for the dual pair $(F, F)$ to minimize the maximum operator norm of reconstruction error operators associated with all possible error locations.
	
	The problem of minimizing reconstruction error has also been investigated from different angles. In \cite{bodm1}, the numerical measure of reconstruction error was taken as the average of the operator norms of error operators across all possible combinations of erasures of a given length. By minimizing this measure, a more robust solution to the erasure problem was sought. Lopez and Han \cite{jerr} discussed a sufficient condition for the canonical dual to be the unique optimal dual, under the operator norm, for one erasure and extended this condition to multiple erasures. Their study also examined the topological properties of the set of all optimal duals of a given frame and provided various characterizations for optimality in this context \cite{jins}.	Further contributions to this field include the work of Pehlivan, Han, and Mohapatra \cite{sali}, who focused on characterizing optimal dual frames that minimize the maximum spectral radii of the error operators pertaining to each possible single erasure. This problem for the case of two erasures has been analysed in \cite{peh, dev}. A similar problem with the measure of the error operator being numerical radius is analysed in \cite{ara} and that with the average of its operator norm and its numerical radius is studied in \cite{deep}. In \cite{bodm1}, the authors introduced a new measure for the error operator, averaging over all possible error positions instead of minimizing the maximum error. In this paper, we adopt the same measurement for the error operator, but our problem differs significantly. Specifically, we study all pairs of frames and their duals, rather than the restricted class of Parseval frames and their canonical duals. We also provide characterizations of optimal dual pairs in the context of graph theory for spectral radius and operator norm as the error measurement.
	
	Given a probability distribution for error locations in frame expansions, \cite{leng3} introduces a weight number sequence to formulate the error operator and establish conditions for dual frame optimality. Additionally, simpler sufficient conditions are provided in \cite{li1}. The search for an optimal dual pair among all Parseval frames and their canonical dual frames is discussed in \cite{leng, li}, while \cite{shan} analyzed both the optimal dual pair among all pairs and the optimal dual frame for a given frame. 
	
	Graph theory is a well-established branch of mathematics known for its rich theoretical foundation and diverse practical applications. It plays an essential role in various real-life scenarios, such as measuring the connectivity of communication networks, facilitating hierarchical clustering, ranking hyperlinks in search engines, and aiding in image segmentation (see \cite{nisan, biggs, shi}). The	concept of frames generated by graphs in finite-dimensional Hilbert spaces is studied in \cite{deep1}. Spectrally optimal dual frames (in short SOD) of a given frames generated by graphs is studied in \cite{deep2}. For those interested in a comprehensive study of matrices associated with simple graphs and other graph classes, the works referenced in \cite{bapat, bapat1, guo, mehat, zas}  offer detailed insights and further reading. In this paper, we aim to investigate the existence and characterization of a dual pair that minimizes, among all dual pairs, the  error operator’s measure under spectral radius and operator norm as the error measurement. We also give a graph theoretic approach, using Laplacian matrix of a graph and Gramian matrix of a frame.
	
The outline of this paper is as follows. In Section 2 of this paper is devoted to the preliminaries and crucial inequalities that are useful for the subsequent sections, Sections 3 and 4 address the problem of using spectral radius and operator norm as error measurements, respectively. More specifically, in Section 3, we present equivalent conditions for determining the optimality of dual pairs under the spectral radius for one and two erasures. A similar problem, in connection with the operator norm for one erasure, is discussed in Section 4. In Section 5, we study spectrally optimal dual pairs generated by graphs for 1 and 2 erasures. We demonstrate that the canonical dual frames of frames generated by connected graphs are the unique spectrally optimal dual frames for 1-erasure. We first compute the spectral radius of the error operator for 1 and 2 erasures of frames generated by graphs. We then establish certain necessary conditions under which a dual pair becomes optimal under both the spectral radius and operator norm.
	
	\section{Preliminaries }
	
	Let $\mathcal{H}$  be an $n-$dimensional Hilbert space.  A finite sequence $F= \{f_i\}_{i=1}^N$ of vectors  in $\mathcal{H}$ is called a \textit{frame} for $\mathcal{H}$ if there exist constants $A,B >0$ such that
	$$\displaystyle{A \left\| f \right\|^2\leq \sum_{i=1}^N\big|\langle f,f_i\rangle\big|^2 \leq B\left\| f \right\|^2}, \,\text{for all}\; f\in \mathcal{H}.$$
	
	\noindent
	The constants $A$ and $B$ are called lower and upper frame bounds respectively.  The  \textit{optimal lower frame bound} is the supremum over all lower frame bounds and the \textit{optimal upper frame bound} is the infimum over all upper frame bounds. If $A=B,$ i.e.  $\displaystyle{\sum_{i=1}^N\big| \langle f,f_i \rangle \big|^2 = A\left\| f \right\|^2},$ for all $f\in \mathcal{H},$  then $\{f_i\}_{i=1}^N$ is called a \textit{tight frame}.
	If $A=B=1,$ then $\{f_i\}_{i=1}^N$  is called a \textit{Parseval frame}. A frame $F$ is said to be \textit{uniform frame} if norm of each frame vector are same and if additionally $\|f_i\| = 1$ for every $i$, then $F$ is called \textit{unit norm} frame. Given a frame $F$ for $\mathcal{H},$ the linear mapping $\Theta_F:\mathcal{H} \to  {\mathbb{C}^N}, $ defined by $$\Theta_F(f)= \{\langle f,f_i \rangle\}_{i=1}^N $$  is called the  \textit{analysis operator} corresponding to $F$ and its adjoint operator   $ \Theta_{F}^*: \mathbb{C}^N \to \mathcal{H}, $  given by
	$$\Theta_{F}^*\left(\{c_i\}_{i=1}^N\right) =  \sum_{i=1}^N {c_i f_i} $$ is called the \textit{synthesis operator or preframe operator} associated with the frame $F$.
	The operator $S_F : \mathcal{H}\to \mathcal{H} $  defined by
	$$S_{F}f= \Theta_{F}^{*} \Theta_{F} f = \sum_{i=1}^N \langle f,f_i \rangle f_i $$
	is called the \textit{frame operator} associated with $F.$
	It is well known that $S_F$ is a positive, self-adjoint and invertible operator on $\mathcal{H}.$	A frame $G=\{g_i\}_{i=1}^N$ for $\mathcal{H}$ is called a dual frame of $F= \{f_i\}_{i=1}^N$ if every element $f \in  \mathcal{H}$ can be written as
	\begin{align} \label{eqn2point1reconstruction}
		f = \sum_{i=1}^N\langle f,f_i \rangle g_i,  \;\;   \forall f\in \mathcal{H}.
	\end{align}
	This implies that $\Theta_{G}^* \Theta_{F} = I$ and hence, $\Theta_{F}^* \Theta_{G} = I.$ So, we also have $	f = \sum_{i=1}^N\langle f,g_i \rangle f_i,  \;\;   \forall f\in \mathcal{H}.$ In other words, $F$ is a dual of $G$ if and only if $G$ is a dual of $F.$ Such a pair $(F,G)$  of $N$ elements each is called an $(N,n)$ dual pair for the $n-$dimensional Hilbert space $\mathcal{H}.$ It is known that for any frame $F,$ the frame $\{S_{F}^{-1}f_i\}_{i=1}^{N}$ is  a dual frame of $F$ and is called the canonical or standard dual frame of $F$. In fact, it is the only dual for a frame $F,$ when $F$ is a basis.  On the other hand, if $F$ is not a basis, then  there exist infinitely many dual frames  $G$ for $F$ and every dual frame $G=\{g_i\}_{i=1}^N$ of F is of the form $G=\{S_{F}^{-1}f_i +u_i\}_{i=1}^N,$ where the sequence $\{u_i\}_{i=1}^N$ satisfies $$\sum_{i=1}^N \langle f,u_i \rangle f_i = \sum_{i=1}^N \langle f,f_i \rangle u_i  = 0, \;\;   \forall f\in \mathcal{H}.$$
	
	\noindent
	Another interesting and useful relation associated with a dual pair $(F,G)$ is 
	$\displaystyle{\sum_{i=1}^N\langle f_i,g_i \rangle} = tr(\Theta_G \Theta_{F}^*)= tr (\Theta_{F}^* \Theta_G)=  tr(I)=n.$ We aim to classify certain dual pair as equivalent. Given dual pair $(F,G)$ and $(F',G'),$ we define them as equivalent if there exists a unitary operator $U$ (orthogonal in the real case) such that $(F',G') = (UF,UG).$	For a detailed study on frames, we refer to \cite{ole}. \\
	Another important operator associated with a frame is the \textit{Gramian operator}. For a frame $F= \{f_i\}_{i=1}^N,$ the Gramian operator $\mathcal{G}: \mathbb{C}^N \to  \mathbb{C}^N $ is defined as $\mathcal{G} = \Theta_{F} \Theta_{F}^*.$ The matrix representation of the Gramian of a frame $F= \{f_i\}_{i=1}^N$ is called the \textit{Gramian matrix} and is defined by $\mathcal{G}= \begin{pmatrix}
		\langle f_1,f_1 \rangle & \langle f_2,f_1 \rangle  & \cdots  &\langle f_N,f_1 \rangle\\
		\langle f_1,f_2 \rangle & \langle f_2,f_2 \rangle  & \cdots  &\langle f_N,f_1 \rangle \\
		\vdots &\vdots  &\vdots &\vdots  \\
		\langle f_1,f_N \rangle & \langle f_2,f_N \rangle  & \cdots  &  \langle f_N,f_N \rangle \\
		
	\end{pmatrix}.$\\~\\

	\noindent
	The concept of treating frames as codes involves considering an original vector $f \in \mathcal{H}$  and an $(N,n)-$frame $F$ with analysis operator $\Theta_F,$ the vector $\Theta_F f$ is viewed as an encoded version of $f,$ which can then be transmitted to a receiver and subsequently decoded by applying $\Theta_{G}^*,$ where $G$ is a dual of $F.$ During data transmission process, it is possible that some frame coefficients in \eqref{eqn2point1reconstruction} may be lost, corrupted, or delayed to the extent that it becomes necessary to reconstruct $f$ using the received components. In this scenario, the received vector can be represented as $\Theta_{G}^* (I_{N\times N} - D) \Theta_{F} f,$ where $D$ is a diagonal matrix consisting of $m$ ones and $N-m$ zeros. The ones in $D$ correspond to the lost components of $\Theta_{G}^* \Theta_{F} f,$ effectively representing the coordinates of $\Theta_{G}^* \Theta_{F} f,$ that have been erased. If errors occur in \textit{`m' positions}, then the error operator is defined by
	$$E_{\Lambda}f:= \Theta_{G}^*D\Theta_{F} f=\sum_{i\in\Lambda} \langle f,f_i \rangle g_i, $$
	where $\Lambda $ is the set of indices corresponding to the erased coefficients, $D$ is an $N\times N $ diagonal matrix with  diagonal elements $d_{ii}= 1$ for $i\in \Lambda$ and 0 otherwise. The error for a dual pair $(F,G)$ for $m-$erasures is defined by 
	$$  \bigg\{ \mathcal{M} \left(\Theta_{G}^*D\Theta_{F}  \right): D \in \mathcal{D}_m \bigg\}, $$
	where where $\mathcal{M}$ is an appropriate measure of the error operator and $\mathcal{D}_m $ is the set of all diagonal matrices with $'m'$ nonzero entries(1 in $i$'th position) and  zeroes in $N-m$ entries on the main diagonal. The goal is to characterize $(F,G)$ dual pair which gives the minimum error. \\
	
	\begin{defn}
		Let $F = \{f_i\}_{i=1}^N$ be a frame for $\mathcal{H}$ and $G = \{g_i\}_{i=1}^N$ be a dual of $F.$ The $(N,n)$ dual pair $(F,G)$ is  called $1-$uniform if  $\langle f_i, g_i \rangle = c,\,\forall  1 \leq i \leq N,$ where $c$ is a constant. 
	\end{defn}
\begin{rem}\label{rem3point2}
	It may be noted that if a $(N,n)$ dual pair $(F,G)$  is $1-$uniform, then the constant $c$ turns out to be $\frac{n}{N}$, for $n= \sum\limits_{i=1}^N \langle f_i,g_i \rangle= cN.$
\end{rem}

		\begin{defn}
		
		A  $1-$uniform $(N,n) $dual pair $(F,G)$ is called a \textit{ $2-$uniform dual pair} if  $\langle f_i ,g_j \rangle \langle f_j , g_i \rangle = c,$  for $1\leq i \neq j \leq N$ and for some constant $c.$ 
	\end{defn}
\begin{thm}\label{thm2point4minq}
	Let $a$ be a positive real number, not equal to $1,$ and $m$ be a rational number, then $a^m -1 > \text{or} < m(a-1),$ according as $m$ does not or does lie between $0$ and $1$.
\end{thm}

	\begin{lem}\label{lemmma2point4}
		Suppose $a_1,a_2,\ldots,a_s$be positive real numbers and  $p_1,p_2,\ldots,p_s$ be positive real numbers. Then for any $m \in \mathbb{Q},$
		$$ \frac{p_1 a_{1}^m + p_2 a_{2}^m+ \cdots + p_s a_{s}^m}{p_1+p_2+ \cdots + p_s}\geq \text{or} \leq \left(\frac{p_1 a_1 + p_2 a_2 + \cdots +p_s a_s}{p_1 + p_2 + \cdots + P_s} \right)^m$$
		according as $m$ does not or does lie between $0$ and $1$ and the equality hold if and only if all $a_i$ are equal.
	\end{lem}
\begin{proof}
	First consider the case when all $a_i$ are not equal. Let us take $q_i = \frac{p_i}{p_1 + p_2 + \cdots + P_s},\,1 \leq i \leq s.$ Then $\sum\limits_{i=1}^s q_i = 1.$ Let $\sum\limits_{i=1}^s q_i a_i = k.$ Then $\frac{a_1}{k}, \frac{a_2}{k},\cdots, \frac{a_s}{k}$ are all positive and  not all of them are equal to $1.$ Thus by Theorem\ref{thm2point4minq},  $\left(\frac{a_i}{k}\right)^m -1 > \text{or} < m \left(\frac{a_i}{k} - 1\right),\;1\leq i \leq s,$ according as $m$ does not or does lie between $0$ and $1.$ Therefore, $q_i\left(\frac{a_i}{k}\right)^m - q_i > \text{or} < mq_i \left(\frac{a_i}{k} - 1\right),$ since $q_i > 0.$ Considering all such $s$ relations for $1 \leq i \leq s,$ adding we have
	$$\sum\limits_{i=1}^s q_i\left(\frac{a_i}{k}\right)^m - \sum\limits_{i=1}^s q_i > \text{or} < m\left[ \dfrac{\sum\limits_{i=1}^s q_i a_{i}^m}{k} - \sum\limits_{i=1}^s q_i \right],$$
	according as $m$ does not or does lie between $0$ and $1.$ This gives, 
	\begin{align*}
		&\;\;\;\;\;\;\;\dfrac{\sum\limits_{i=1}^s q_i a_{i}^m}{k^m} -1 > \text{or} < 0 \\& \text{or},\; \sum\limits_{i=1}^s q_i a_{i}^m >\text{or} < k^m = \left(  \sum\limits_{i=1}^s q_i a_i\right)^m \\& \text{or},\;\frac{p_1 a_{1}^m + p_2 a_{2}^m+ \cdots + p_s a_{s}^m}{p_1+p_2+ \cdots + p_s} > \text{or} < \left(\frac{p_1 a_1 + p_2 a_2 + \cdots +p_s a_s}{p_1 + p_2 + \cdots + P_s} \right)^m,
	\end{align*}
according as $m$ does not or does lie between $0$ and $1.$\\
It is straightforward to see that when all $a_i$ are equal, then $\dfrac{p_1 a_{1}^m + p_2 a_{2}^m+ \cdots + p_s a_{s}^m}{p_1+p_2+ \cdots + p_s} = \left(\dfrac{p_1 a_1 + p_2 a_2 + \cdots +p_s a_s}{p_1 + p_2 + \cdots + P_s} \right)^m.$ Conversely, if $\dfrac{p_1 a_{1}^m + p_2 a_{2}^m+ \cdots + p_s a_{s}^m}{p_1+p_2+ \cdots + p_s} = \left(\dfrac{p_1 a_1 + p_2 a_2 + \cdots +p_s a_s}{p_1 + p_2 + \cdots + P_s} \right)^m,$ then it is evident that all $a_i$ are equal. If this were not the case, then using the same argument as above, we can prove that the equality does not hold.
\end{proof}
	
	\section{Spectrally optimal dual pairs}
	In this section, we assume the measure $\mathcal{M}$ of the error operator to be its spectral radius $\rho.$ Here, we look for that dual pair which minimize the error by taking spectral radius as the error measurement. Let us define for $p >1,$
 \begin{align*}
		&\mathcal{E}_{1}^p (F,G) := \left\{\frac{1}{N} \sum\limits_{D \in \mathcal{D}^{(1)}}\left( \rho( \Theta_{G}^*D\Theta_{F})\right)^p\right\}^{\frac{1}{p}} \;\\
		& \delta_{1}^p := \inf\left\{ \mathcal{E}_{1}^p (F,G) : (F,G)\, \text{is a $(N,n)$ dual pair }   \right\}\\
     & \mathcal{E}_{1}^p := \left\{(F',G'): \mathcal{E}_{1}^p (F',G') =  \delta_{1}^p  \right\}
	\end{align*}
Every element in $ \mathcal{E}_{1}^p$ is called a $1-$ersaure spectrally optimal dual pair. For  $1<k \leq N,$ let us define:

	 \begin{align*}
		&\mathcal{E}_{k}^p (F,G) := \left\{\frac{1}{\binom{N}{k}} \sum\limits_{D \in \mathcal{D}^{(1)}}\left( \rho( \Theta_{G}^*D\Theta_{F})\right)^p\right\}^{\frac{1}{p}} \,\;\\
		& \delta_{k}^p := \inf\left\{ \mathcal{E}_{k}^p (F,G) : (F,G)\in \mathcal{E}_{k-1}^p    \right\}\\
     & \mathcal{E}_{k}^p := \left\{(F',G'): (F',G') \in \mathcal{E}_{k-1}^p\, \text{and}\,  \mathcal{E}_{k}^p (F',G') =  \delta_{k}^p  \right\}
	\end{align*}

 Every element in $ \mathcal{E}_{k}^p$ is called a $k-$ersaure spectrally optimal dual pair.

\noindent
For a dual pair $(F,G),$ if the error occurs in the $i^{th}$ position, then $\rho\left(\Theta_{G}^*D\Theta_{F}\right) = \left|\langle f_i, g_i \rangle \right| $ and so,
	\begin{equation}\label{equation2}
		\mathcal{E}_{1}^p (F,G) = \left\{\frac{1}{N} \sum\limits_{i=1}^N \left|\langle f_i, g_i \rangle \right|^p\right\}^{\frac{1}{p}}
	\end{equation}\\
	 We shall now give a characterization for an $(N,n)$ dual pair to be a $1-$erasure spectrally optimal. 
	
	\begin{prop}\label{prop3point1}
		The  value of $ \delta_{1}^p$ is $\frac{n}{N}$. Moreover, an $(N,n)$ dual pair $(F,G) $ is a $1-$erasure spectrally optimal  if and only if it is $1-$uniform.
	\end{prop}
	\begin{proof}
For a $(N,n)$ dual pair $(F,G),$ it is easy to see that 
\begin{align}\label{eqn3point3}
    \mathcal{E}_{1}^p (F,G) = \left\{\frac{1}{N} \sum\limits_{i=1}^N \left|\langle f_i, g_i \rangle \right|^p\right\}^{\frac{1}{p}} \geq \left( \left(\frac{\sum\limits_{i=1}^N \left|\langle f_i, g_i \rangle \right|}{N}   \right)^p \right)^{\frac{1}{p}} = \frac{\sum \limits_{i=1}^N\left|\langle f_i, g_i \rangle \right|}{N} \geq \frac{n}{N},
\end{align}
the second inequality follows from the application of Lemma\ref{lemmma2point4}.
 Suppose there exists a uniform Parseval frame $F'$. Then, we will have  $\mathcal{E}_{1}^p (F',F') = \frac{n}{N} $ and hence, $ \mathcal{E}_{1}^p  = \frac{n}{N}.$ The existence of such a Parseval frame $F'$  is guaranteed by Theorem 2.1 in \cite{casa2}. In particular, we take $S$ to be the identity operator, which has only $1$ as its eigenvalues, and $a_i = \sqrt{\frac{n}{N}}, \; \forall i.$ As $\sum\limits_{i=1}^k a_{i}^2 = \frac{kn}{N},$ for $ 1\leq k \leq n$ and  $\sum\limits_{i=1}^N a_{i}^2 = n,$ there exists a Parseval frame $\{f'_i\}_{i=1}^N$ such that $\| f'_i\|^2 = \frac{n}{N},$ for all $i.$ Therefore, $\mathcal{E}_{1}^p (F',F') = \frac{n}{N}$ and hence, 
$\delta_{1}^p = \frac{n}{N}.$ 
\par If $(F,G)$ is a $1-$uniform dual pair, then $\langle f_i, g_i \rangle = \frac{n}{N},\;\forall i. $ Thus, by \eqref{equation2}, we have $\mathcal{E}_{1}^p (F,G) = \frac{n}{N}$ and hence $(F,G) $  is a $1-$erasure spectrally optimal dual pair. Conversely, suppose  $(F,G) $  is a $1-$erasure spectrally optimal dual pair. We then have $\mathcal{E}_{1}^p (F,G) = \frac{n}{N} .$ Thus by using the expression of $\mathcal{E}_{1}^p (F,G)$ as in \eqref{equation2}, $\left\{\frac{1}{N} \sum\limits_{i=1}^N \left|\langle f_i, g_i \rangle \right|^p\right\}^{\frac{1}{p}} = \left( \left(\frac{\sum\limits_{i=1}^N \left|\langle f_i, g_i \rangle \right|}{N}   \right)^P \right)^{\frac{1}{p}} = \frac{\sum\limits_{i=1}^N \left|\langle f_i, g_i \rangle \right|}{N} = \frac{n}{N}.$ By Lemma \ref{lemmma2point4}, this is possible if and only if $\left|\langle f_i, g_i \rangle \right| =  \frac{n}{N},\,\forall i.$
 Now, $ n = \displaystyle{\sum_{j=1}^N \langle f_j , g_j \rangle \leq \sum_{j=1}^N |\langle f_j , g_j \rangle |  = n } .$  It now suffices to show that $\langle f_j, g_j \rangle \geq 0,\,\forall\, j.$ Clearly,
			$$\sum\limits_{j=1}^N \langle f_j , g_j \rangle = n =  \sum\limits_{j=1}^N |\langle f_j , g_j \rangle | .$$
			Let $\langle f_j , g_j \rangle = a_j + ib_j,\, 1\leq j \leq N,$ where $a_j,b_j \in \mathbb{R}.$ Then, $\sum\limits_{j=1}^N a_j = n = \sum\limits_{j=1}^N \sqrt{a_j^2 +b_j^2,}$ which shows that $b_j = 0$  and also $a_j \geq 0,\, \forall\, j.$
\end{proof}

\begin{cor}
		For a frame $F$ in $\mathcal{H},$\, $S_{F}^{-1}F$  is a uniform Parseval frame if and only if \;$\left(F, S_{F}^{-\frac{1}{2}}F\right) \in \mathcal{E}_{1}^p.$
	\end{cor}
\begin{proof}
    It is easy to see that if  $S_{F}^{-\frac{1}{2}}F$  is a uniform Parseval frame, then $\left\| S_{F}^{-\frac{1}{2}}f_i \right\|^2 = \frac{n}{N}, \;\forall i,$ Since, $\mathlarger{\sum\limits_{i=1}^N }\left\|  S_{F}^{-\frac{1}{2}}f_i  \right\|^2 = \sum\limits_{i=1}^N \langle f_i, S_{F}^{-1}f_i \rangle = n. $ Therefore by \eqref{equation2}, $\mathcal{E}_{1}^p (F,S_{F}^{-1}F) = \left\{\frac{1}{N} \sum\limits_{i=1}^N \left\|  S_{F}^{-\frac{1}{2}}f_i  \right\|^{2p}\right\}^{\frac{1}{p}} = \frac{n}{N}. $ Hence, by Proposition \ref{prop3point1}, $(F, S_{F}^{-1}F) \in \mathcal{E}_{1}^p .$ Conversely, if  $(F, S_{F}^{-1}F) \in \mathcal{E}_{1}^p,$ then by Proposition \ref{prop3point1},  $(F, S_{F}^{-1}F)$ is a $1-$uniform dual pair and hence, $\left\|S_{F}^{-\frac{1}{2}}f_i\right\|^2 = \langle f_i, S_{F}^{-1}f_i \rangle = \frac{n}{N}, \;\forall i.$ 
\end{proof}

	\noindent From the proof of Proposition \ref{prop3point1}, we can also conclude the following:
	
	\begin{thm}\label{thm3point6}
	There exists a Parseval frame $F$ in $\mathcal{H}$ such that $(F,F) \in \mathcal{E}_{1}^p . $ 
	\end{thm}

It is established that $G$ is a dual of a frame $F$ if and only if $UG$ is a dual of the frame $UF,$ where $U$ is a unitary operator. The following theorem demonstrates that the collection of all $1-$erasure spectrally optimal dual pairs remains invariant under the action of a unitary operator.
\begin{thm}\label{thm3point7}
	For a unitary operator $U$ on $\mathcal{H}$, a dual pair $(F,G)$ is a $1-$erasure spectrally optimal dual pair if and only if $(UF,UG)$  is a $1-$erasure spectrally optimal dual pair. 
\end{thm}
\noindent	The proof of the above theorem is straightforward, as
 $$\mathcal{E}_{1}^p (F,G) = \left\{\frac{1}{N} \sum\limits_{i=1}^N \left|\langle f_i, g_i \rangle \right|^p\right\}^{\frac{1}{p}} = \left\{\frac{1}{N} \sum\limits_{i=1}^N \left|\langle Uf_i, Ug_i \rangle \right|^p\right\}^{\frac{1}{p}} = \mathcal{E}_{1}^p (UF,UG) .$$

\vspace{1.5mm}
\par Suppose the frames under consideration have exactly $n$ elements. Then, such  frames have only one dual, namely the canonical dual. So, the analysis in this case boils down to the collection of all $(n,n)$ dual pairs of the form $(F,S_{F}^{-1}F).$ For any $(n,n)$ dual pair $(F,S_{F}^{-1}F),$ we have  
$$\mathcal{E}_{1}^p (F,S_{F}^{-1}F) = \left\{\frac{1}{n} \sum\limits_{i=1}^n \left|\langle f_i, S_{F}^{-1}f_i \rangle \right|^p\right\}^{\frac{1}{p}} = \left\{\frac{1}{n} \sum\limits_{i=1}^n \left\| S_{F}^{-\frac{1}{2}}f_i  \right\|^{2p}\right\}^{\frac{1}{p}}=1,$$
as  $S_{F}^{-\frac{1}{2}}F$ is both a Parseval frame as well as a basis and it is well-known that  a basis in a Hilbert space is a Parseval frame if and only if it is orthonormal. Therefore, $\mathcal{E}_{1}^p  = 1.$ Thus, every $(n,n)$ dual pair of the form $(F,S_{F}^{-1}F),$  is a $1-$erasure spectrally optimal dual pair. \\~\\
	We shall also provide a very useful expression for $\mathcal{E}_{2}^p (F,G).$ Consider
\begin{eqnarray*}
\mathcal{E}_{2}^p (F,G) = \left\{\frac{1}{\binom{N}{2}} \sum\limits_{D \in \mathcal{D}^{(2)}}\left( \rho( \Theta_{G}^*D\Theta_{F})\right)^p\right\}^{\frac{1}{p}} = \left\{\frac{1}{\binom{N}{2}} \sum\limits_{D \in \mathcal{D}^{(2)}}\left( \rho( D\Theta_{G}^*\Theta_{F})\right)^p\right\}^{\frac{1}{p}}.
\end{eqnarray*}
Taking  $D = diag [0,0,\ldots,0,1,0,\ldots,0,1,0,\ldots,0],$ the matrix representation $A$ of $D\Theta_{F}\Theta_{G}^*$ with respect to the standard orthonormal basis of $\mathbb{C}^N$ is

	$$	\begin{pmatrix}
		0 & 0  & \cdots &0 & \cdots &0 & \cdots &0\\
		\vdots &\vdots  &\vdots &\vdots &\vdots&\vdots &\vdots &\vdots \\
		0 & 0  & \cdots &0 & \cdots &0 & \cdots &0 \\
		\alpha_{1i} & \alpha_{2i} & \cdots &\alpha_{ii} & \cdots &\alpha_{ji} & \cdots &\alpha_{Ni} \\
		0 & 0  & \cdots &0 & \cdots &0 & \cdots &0\\
		\vdots &\vdots  &\vdots &\vdots &\vdots&\vdots &\vdots &\vdots \\
		0 & 0  & \cdots &0 & \cdots &0 & \cdots &0 \\
		\alpha_{1j} & \alpha_{2j} & \cdots &\alpha_{ij} & \cdots &\alpha_{jj} & \cdots &\alpha_{Nj} \\
		0 & 0  & \cdots &0 & \cdots &0 & \cdots &0\\
		\vdots &\vdots  &\vdots &\vdots &\vdots&\vdots &\vdots &\vdots \\
		0 & 0  & \cdots &0 & \cdots &0 & \cdots &0 \\
	\end{pmatrix},$$
	where $\alpha_{ij}:= \langle g_i , f_j \rangle, \forall\,i,j. $
	\noindent
	The characteristic polynomial of the above matrix is
	\begin{align*}
		\det(A-xI) &= (-x)^{i-1}(-x)^{N-j} \left|\begin{array}{ccccccc}
			-x +  \alpha_{ii} &\alpha_{(i+1) i} &\cdots &\alpha_{(j-1)i} &\alpha_{ji}\\
			0 & -x  &\cdots& 0 &0\\
			\vdots &\vdots  &\vdots &\vdots &\vdots \\
			0 & 0  &\cdots& -x &0\\
			\alpha_{ij} & \alpha_{(i+1)j} & \cdots &\alpha_{(j-1)j} & -x + \alpha_{jj} \\
		\end{array}\right| \\ &= (-x)^{N-2} \left((x- \alpha_{ii})(x-\alpha_{jj}) -  \alpha_{ij} \alpha_{ji} \right).
	\end{align*}
	\noindent
	Therefore, the eigenvalues of $A$ are $0,\, \dfrac{ \alpha_{ii} + \alpha_{jj} \pm \sqrt{( \alpha_{ii}- \alpha_{jj})^2 + 4 \alpha_{ij} \alpha_{ji} }}{2}. $  Hence, $$ \rho( \Theta_{G}^*D\Theta_{F}) = \left| \frac{\alpha_{ii} + \alpha_{jj} \pm \sqrt{(\alpha_{ii} - \alpha_{jj})^2 + 4  \alpha_{ij}\alpha_{ji} }}{2} \right| = \left| \frac{\alpha_{ii} + \alpha_{jj} + \sqrt{(\alpha_{ii} - \alpha_{jj})^2 + 4 \alpha_{ij}\alpha_{ji} }}{2} \right|.$$ Thus, for a dual pair $(F,G) \in 	\mathcal{E}_{1}^p,$ we have
	\begin{equation} \label{eqn3point3}
		\mathcal{E}_{2}^p (F,G) = \left\{\frac{1}{\binom{N}{2}} \sum\limits_{i \neq j}\left|  \frac{n}{N} +  \sqrt{\alpha_{ij} \alpha_{ji}} \right|^p\right\}^{\frac{1}{p}} .
	\end{equation}

\begin{lem}\label{lem3point5}
	Let $E = \left\{(\alpha_1,\alpha_2, \ldots, \alpha_r) \in \mathbb{R}^r : \sum\limits_{i=1}^r \alpha_i = c\right\},$ where $c\geq 0.$ Then for any 
	$a>0$ and $p>1,$ 
	$$ \inf\limits_{(\alpha_1,\alpha_2, \ldots, \alpha_r) \in E} \sum\limits_{i=1}^r \left| a+ \sqrt{\alpha_i}  \right|^p = r\left(a+ \sqrt{\frac{c}{r}}  \right)^p.$$
\end{lem}

\begin{proof}
	For any ${(\alpha_1,\alpha_2, \ldots, \alpha_r) \in E}, $ it can be easily seen,  using Lemma\ref{lemmma2point4}, that 
	\begin{eqnarray}\label{eqn3point5}
		\sum\limits_{i=1}^r \left| a+ \sqrt{\alpha_i}  \right|^p \geq r \left( \frac{\sum\limits_{i=1}^r \left| a+ \sqrt{\alpha_i}  \right|}{r}  \right)^p
	\end{eqnarray}
and the equality hold if and only if each $\left| a+ \sqrt{\alpha_i}  \right|$ is equal. In this case, $\alpha_i = \frac{c}{r},\,\forall i.$ Therefore, from \eqref{eqn3point5}, we have
\begin{eqnarray*}
	\sum\limits_{i=1}^r \left| a+ \sqrt{\alpha_i}  \right|^p \geq r\left(a+ \sqrt{\frac{c}{r}}  \right)^p. 
\end{eqnarray*}
Hence, the result follows.
\end{proof}

\begin{rem}
	Note that in Lemma \ref{lem3point5}, if for $(\alpha_1,\alpha_2, \ldots, \alpha_r) \in E,$ infimum value attain, then $\alpha_i \geq 0,$ for all $i.$ Suppose not. Then first consider the case when only one term is negative, say $\alpha_j <0$ and $\alpha_i \geq 0,\;\forall i \neq j.$ Then,
	\begin{eqnarray*}
			\sum\limits_{i=1}^r \left| a+ \sqrt{\alpha_i}  \right|^p = 	\sum\limits_{i \neq j}  \left( a+ \sqrt{\alpha_i}  \right)^p + \left(a^2 -\alpha_j \right)^{\frac{p}{2}} > 	\sum\limits_{i \neq j}  \left( a+ \sqrt{\alpha_i}  \right)^p + a^p.
	\end{eqnarray*}
Note that, $\sum\limits_{i \neq j} \alpha_i = c+ |\alpha_j| > c.$ Thus, there exists a tuple  $(\alpha'_1,\alpha'_2, \ldots, \alpha'_r) \in \mathbb{R}_{+}^r $ with $\alpha'_j =0,\,\sum\limits_{i=1}^r \alpha'_i = c$ and $\sum\limits_{i \neq j}  \left( a+ \sqrt{\alpha_i}  \right)^p > \sum\limits_{i \neq j}  \left( a+ \sqrt{\alpha'_i}  \right)^p.$ Hence, $\sum\limits_{i=1}^r \left| a+ \sqrt{\alpha_i}  \right|^p > \sum\limits_{i=1}^r \left| a+ \sqrt{\alpha'_i}  \right|^p,$ which is not possible. Using the same argument repeatedly at most $r-1$ times we can conclude that $\alpha_i \geq 0,\;1 \leq i \leq r.$
\end{rem}
\noindent
We shall now give a characterization for a $2-$erasure spectrally optimal dual pair. Recall that if $(F,G) \in \mathcal{E}_{1}^p,$ then $\mathcal{E}_{2}^p (F,G) = \left\{\dfrac{1}{\binom{N}{2}} \sum\limits_{i \neq j}\left|  \frac{n}{N} +  \sqrt{\alpha_{ij} \alpha_{ji}} \right|^p\right\}^{\frac{1}{p}} .$

\begin{thm}\label{thm3point7} 
	Let $\mathcal{H}$ be an $n$ dimensional real Hilbert space. Then, $\delta_{2}^p \geq \frac{n}{N} + \sqrt{\frac{nN - n^2}{N^2(N-1)}}. $ Moreover, if a $2-$uniform $(N,n)$ dual pair exists, then  $\delta_{2}^p = \frac{n}{N} + \sqrt{\frac{nN - n^2}{N^2(N-1)}}. $
\end{thm}

\begin{proof}
	For a dual pair  $(F,G) \in \mathcal{E}_{1}^p,$ by \eqref{eqn3point3}, we have 
	$$ 	\mathcal{E}_{2}^p (F,G) = \left\{\frac{1}{\binom{N}{2}} \sum\limits_{i \neq j}\left|  \frac{n}{N} +  \sqrt{\alpha_{ij} \alpha_{ji}} \right|^p\right\}^{\frac{1}{p}} .  $$
	Now, using Lemma \ref{lem3point5},
	\begin{align*} 
		\mathcal{E}_{2}^p (F,G) &= \left\{\frac{1}{\binom{N}{2}} \sum\limits_{i \neq j}\left|  \frac{n}{N} +  \sqrt{\alpha_{ij} \alpha_{ji}} \right|^p\right\}^{\frac{1}{p}} \\& \geq \left(\frac{1}{\binom{N}{2}} \times \binom{N}{2} \left( \frac{n}{N} + \sqrt{\frac{nN - n^2}{N^2(N-1)}}\right)^p \right)^{\frac{1}{p}} \\& = \frac{n}{N} + \sqrt{\frac{nN - n^2}{N^2(N-1)}}
	\end{align*}
Therefore, $\delta_{2}^p = \inf\left\{ \mathcal{E}_{2}^p (F,G) : (F,G)\in \mathcal{E}_{1}^p    \right\} \geq \frac{n}{N} + \sqrt{\frac{nN - n^2}{N^2(N-1)}}. $\\
Moreover, if a $2-$uniform $(N,n)$ dual pair $(F',G')$ exist,  then $\alpha'_{ij} \alpha'_{ji} = \frac{nN - n^2}{N^2(N-1)},\,\forall i \neq j. $ Therefore, $	\mathcal{E}_{2}^p (F',G') =   \frac{n}{N} + \sqrt{\frac{nN - n^2}{N^2(N-1)}}$ and hence, $\delta_{2}^p = \frac{n}{N} + \sqrt{\frac{nN - n^2}{N^2(N-1)}}.$
\end{proof}

	\section{Optimal dual pairs under operator norm}
In this section, we look for that dual pair which minimizes the error taking operator norm as the error measurement $\mathcal{M}.$ Let us define for $p >1,$
\begin{align*}
	&\mathcal{O}_{1}^p (F,G) := \left\{\frac{1}{N} \sum\limits_{D \in \mathcal{D}^{(1)}} \left\| \Theta_{G}^*D\Theta_{F}\right\|^p\right\}^{\frac{1}{p}} \;\\
	& \Delta_{1}^p  := \inf\left\{ \mathcal{O}_{1}^p (F,G) : (F,G)\, \text{is a $(N,n)$ dual pair }   \right\}\\
	&  \mathcal{O}_{1}^p := \left\{(F',G'): \mathcal{O}_{1}^p (F',G') =  \mathcal{O}_{1}^p \right\}
\end{align*}
Every element in $ \mathcal{O}_{1}^p$ is called a $1-$ersaure  optimal dual pair.\\
For a dual pair $(F,G),$ if the error occurs in the $i^{th}$ position, then $\left\|\Theta_{G}^*D\Theta_{F}\right\| = \|f_i\|\,\| g_i \|  $ and so,
\begin{equation}\label{equation6}
	\mathcal{O}_{1}^p (F,G) = \left\{\frac{1}{N} \sum \limits_{i=1}^N\|f_i\|^p\,\| g_i \| ^p\right\}^{\frac{1}{p}}.
\end{equation}\\
We shall now give a characterization for an $(N,n)$ dual pair to be a $1-$erasure optimal by taking operator norm as the error measurement.

	\begin{prop}\label{prop4point1}
	Let $\mathcal{H}$ be an $n$ dimensional Hilbert space. Then the  value of $ \Delta_{1}^p$ is $\frac{n}{N}$ and  $ \mathcal{O}_{1}^p = \left\{(F,G): \frac{1}{N} \sum\limits_{i=1}^N \left(\|f_i\|\,\| g_i \| \right)^p =  \left(  \frac{n}{N}  \right)^p \right\}$ 
\end{prop}
\begin{proof}
	For an $(N,n) $ dual pair $(F,G),$ it is easy to see that
	\begin{align*}
		\mathcal{O}_{1}^p (F,G) = \left\{\frac{1}{N} \sum \limits_{i=1}^N\|f_i\|^p\,\| g_i \| ^p\right\}^{\frac{1}{p}} \geq \left\{ \left(\frac{\sum\limits_{i=1}^N \|f_i\|\,\| g_i \|}{N}  \right)^p\right\}^{\frac{1}{p}} \geq \frac{\sum\limits_{i=1}^N \left|\langle f_i, g_i \rangle \right|}{N} \geq \frac{n}{N}
	\end{align*}
Suppose there exists a uniform Parseval frame $F'$. Then, we will have  $\mathcal{O}_{1}^p (F',F') = \frac{n}{N} $ and hence, $ \Delta_{1}^p  = \frac{n}{N}.$ Therefore by definition we have 
\begin{align*}
	\mathcal{O}_{1}^p  &= \left\{(F,G): \mathcal{O}_{1}^p (F,G) = \frac{n}{N} \right\} \\& = \left\{(F,G): \left(\frac{1}{N} \sum \limits_{i=1}^N\|f_i\|\,\| g_i \| ^p\right)^{\frac{1}{p}} = \frac{n}{N} \right\} 
\end{align*}
\end{proof}
 Now, we present a necessary condition for a dual pair to be $1-$erasure optimal. Specifically, we examine the criteria that must be satisfied for a dual pair to maintain optimality even when one element is erased. This involves analyzing the structural properties and relationships within the dual pair to ensure minimal loss for $1-$erasure.
 
 \begin{thm}\label{thm4point2}
 	If $(F,G) \in \mathcal{O}_{1}^p $ then, $(F,G)$ is a $1-$uniform $(N,n)$ dual pair for $\mathcal{H}.$ 
 \end{thm}
In order to prove the above theorem, we make use of the following lemma.

\begin{lem}\label{lemma4point3}
	Let $\{a_i\}_{i=1}^s$ be a sequence of non-negative real numbers satisfying $\sum\limits_{i=1}^s a_i= t$ and $\sum\limits_{i=1}^s a^{p}_i= s\left(\frac{t}{s}  \right)^p,$ for some $p>1.$ Then, $a_i = \frac{t}{s},\;1 \leq i \leq s.$
\end{lem}

\begin{proof}
	Using Lemma \ref{lemmma2point4}, it can be easily seen that
	$$s\left(\frac{t}{s}  \right)^p =  \sum\limits_{i=1}^s a^{p}_i \geq s \left(\frac{\sum\limits_{i=1}^s a_i}{s}\right)^p = s\left(\frac{t}{s}  \right)^p.$$ 
	This implies that, $\dfrac{\sum\limits_{i=1}^s a^{p}_i}{s} = \left(\dfrac{\sum\limits_{i=1}^s a_i}{s}\right)^p . $ Thus by Lemma \ref{lemmma2point4}, $a_i = a_j $ for all $i \neq j$ and hence $a_i = \frac{t}{s},\;1 \leq i \leq s.$
\end{proof}

\begin{proof}[\textbf{Proof of Theorem \ref{thm4point2}}]
	Let $(F,G) \in \mathcal{O}_{1}^p .$ Then by Proposition \ref{prop4point1}, we have
	\begin{align*}
		n^p = \left(\sum\limits_{i=1}^N \langle f_i, g_i \rangle \right)^p \leq \left(\sum\limits_{i=1}^N \| f_i\| \, \|g_i \| \right)^p \leq N^{p-1} \sum\limits_{i=1}^N \left( \| f_i\| \, \|g_i \| \right)^p \leq N^{p-1} N \left(\frac{n}{N}\right)^p = n^p.
	\end{align*} 
Consequently, we obtain, $\left(\sum\limits_{i=1}^N \langle f_i, g_i \rangle \right)^p = \left(\sum\limits_{i=1}^N \| f_i\| \, \|g_i \| \right)^p .$ Thus $\sum\limits_{i=1}^N \langle f_i, g_i \rangle = \sum\limits_{i=1}^N \| f_i\| \, \|g_i \| .$ This in turn leads to, $\sum\limits_{i=1}^N \| f_i\| \, \|g_i \| = \sum\limits_{i=1}^N \langle f_i, g_i \rangle \leq \left| \sum\limits_{i=1}^N \langle f_i, g_i \rangle  \right|  \leq \sum\limits_{i=1}^N \left| \langle f_i, g_i \rangle \right| \leq \sum\limits_{i=1}^N \| f_i\| \, \|g_i \|.$ Accordingly, $\sum\limits_{i=1}^N \langle f_i, g_i \rangle = \sum\limits_{i=1}^N \left| \langle f_i, g_i \rangle \right|. $ Let $\langle f_j , g_j \rangle = a_j + ib_j,\, 1\leq j \leq N,$ where $a_j,b_j \in \mathbb{R}.$ Then, $\sum\limits_{j=1}^N a_j = n = \sum\limits_{j=1}^N \sqrt{a_j^2 +b_j^2,}$ which shows that $b_j = 0$  and also $a_j \geq 0,\, \forall\, j.$ So, $\langle f_i, g_i \rangle = \left| \langle f_i, g_i \rangle \right|,\,1\leq i \leq N.$ By proposition \ref{prop4point1}, we have $\sum\limits_{i=1}^N \| f_i\| \, \|g_i \| = n $ and $\sum\limits_{i=1}^N \left(\| f_i\| \, \|g_i \|\right)^p = N\left(\frac{n}{N} \right)^p.$ Then by Lemma\ref{lemma4point3}, we have $\| f_i\| \, \|g_i \|= \frac{n}{N}. $ Thus, $ \langle f_i, g_i \rangle = \frac{n}{N},$ as  $\sum\limits_{i=1}^N \langle f_i, g_i \rangle = \sum\limits_{i=1}^N \| f_i\| \, \|g_i \| .$
 \end{proof}
\noindent
A theorem analogous to Theorem\ref{thm3point7} states as follows.

\begin{thm}\label{thm4point4unitary}
	For a unitary operator $U$ on $\mathcal{H}$, a dual pair $(F,G) \in \mathcal{O}_{1}^p$ if and only if $(UF,UG)  \in \mathcal{O}_{1}^p.$ 
\end{thm}

\par  Suppose the frames under consideration have exactly $n$ elements. Then, such  frames have only one dual, namely the canonical dual. So, the analysis in this case boils down to the collection of all $(n,n)$ dual pairs of the form $(F,S_{F}^{-1}F).$ For any $(n,n)$ dual pair $(F,S_{F}^{-1}F),$ we have  $$\mathcal{O}_{1}^p (F,S_{F}^{-1}F) = \left\{\frac{1}{n} \sum \limits_{i=1}^n\|f_i\|\,\|S_{F}^{-1}f_i  \| ^p\right\}^{\frac{1}{p}} \geq \left(\dfrac{1}{n}\sum \limits_{i=1}^n\left| \langle f_i, S_{F}^{-1}f_i  \rangle\right|^p\right)^{\frac{1}{p}} \geq \dfrac{\sum \limits_{i=1}^n\left| \langle f_i, S_{F}^{-1}f_i  \rangle\right|}{n} \geq 1,$$
by using Lemma\ref{lemmma2point4}. Further, for any Parseval frame $F'=\{f'_i\}_{i=1}^n,$ it is easy to see that $\mathcal{O}_{1}^p (F',S_{F'}^{-1}F') =  \left(\dfrac{\sum \limits_{i=1}^n\|f_i\|^{2p}}{n}\right)^{\frac{1}{p}} =  1 ,$ as  $F$ is both a Parseval frame as well as a basis and it is well-known that  a basis in a Hilbert space is a Parseval frame if and only if it is orthonormal. This leads to $\mathcal{O}_{1}^p  = 1.$ Thus, every $(n,n)$ dual pair of the form $(F,F),$ where $F$ is a Parseval frame, is a $1-$erasure optimal dual pair under operator norm. In fact in this context, it can be inferred that any equal-norm tight frame $\{f_i\}_{i=1}^n$ and its canonical dual serves as an optimal dual pair.

\section{Graph theoretic approach}

Let $\Gamma$ be a simple graph with the vertex set $V(\Gamma)=\{v_1,v_2,…,v_N\}.$ Two vertices $v_i$ and $v_j$ are adjacent if there is an edge connecting them, which we denote as $v_i \sim v_j$. The graph $G$ is a complete graph if  $v_i \sim v_j$ for all $i \neq j.$ The graph $G$ is a complete graph if  $v_i \sim v_j$ for all $i \neq j.$ The degree of a vertex $v_i$ is the count of edges connected to $v_i$ is represented by $d(v_i)$ or simply $d_i$. A vertex $v_i$ is termed a null vertex if $d(v_i)=0.$ A graph $\Gamma$ is described as a regular graph if $d_i=d_j$ for all $1\leq i,j \leq N.$ If $d_i=d_j=r$ for all $i,j,$ then $\Gamma$ is specifically called an $r-$regular graph.\\
Now, we examine some matrices associated with a simple graph $\Gamma.$ The degree matrix of $\Gamma,$ denoted by $\mathcal{D}(\Gamma),$ is an $N \times N$ diagonal matrix where $\mathcal{D}(\Gamma)= \text{diag}(d_1,d_2,\ldots,d_n).$ The adjacency matrix of the graph $\Gamma$, denoted as $\mathcal{A}(\Gamma)=[a_{ij}]_{N \times N}$ is an $N \times N$ matrix defined as:

$$  	a_{ij} = \begin{cases}
	1,\, \text{for}\; v_i \sim v_j \;\text{and}\; i \neq j \\
	0,\,\text{elsewhere}	.
\end{cases}  $$$$
 $$

The Laplacian matrix of a graph $\Gamma$ is defined as $\mathcal{L}(\Gamma):=\mathcal{D}(\Gamma)-\mathcal{A}(\Gamma)$ The Laplacian matrix is positive semi-definite. For a simple graph $\Gamma$ with $N$ vertices and $k$ components, the rank of $L(\Gamma)$ is $N-k.$ For a comprehensive study of matrices associated with simple graphs and other types of graphs, readers can refer to \cite{bapat}. It is well established that $0$ is always an eigenvalue of the Laplacian matrix and is the smallest eigenvalue. If $\Gamma$ is a graph with Laplacian matrix $\mathcal{L}$ and $\nu_1 \geq \nu_2\geq \ldots \geq\nu_k$ are the non-zero eigenvalues of $\mathcal{L},$ then the second smallest eigenvalue, $\nu_{k-1}$, is referred to as the \textit{algebraic connectivity} of $\Gamma.$ \\
Let $F= \{f_i\}_{i=1}^N$ be a frame for $\mathcal{H}$ with Gramian matrix $\mathcal{G}.$ If $\Gamma$ corresponds to a simple graph with Laplacian matrix $\mathcal{L}$ such that $\Gamma=\mathcal{L}$, then $F$ is known as a frame generated by the graph $\Gamma.$ For short, we refer to $F$ as a $\Gamma(N,n)-$frame for $\mathcal{H}.$ Any two frames generated by a simple graph $\Gamma$ are unitarily equivalent \cite{deep1}.
 In \cite{deep2}, the authors employed Laplacian matrices of graphs to construct finite frames. Before delving into this construction, let us first define $\mathcal{L}_\Gamma(N,n)-$frames. Let $\Gamma$ be a simple graph with $N$ vertices and $N-k$ components. Suppose $L$ can be decomposed as $\mathcal{L}= M \textit{diag}(\lambda_1, \lambda_2, \ldots,\lambda_n, 0,\ldots,0)M^*,\,$where $\lambda_1,\lambda_2, \ldots,\lambda_n$ are the non-zero eigenvalues of $L$ and $M$ is a corresponding matrix of eigenvectors of $L.$ Consider $\{e_i\}_{i=1}^N$ to be the standard canonical orthonormal basis of $\mathbb{C}^N$. Define $B$ as $B= \textit{diag}\left(\sqrt{\lambda_1}, \sqrt{\lambda_2}, \ldots,\sqrt{\lambda_n}\right)$, where $M_1$ is a submatrix of $M$ consisting of the first $n$ columns. Then $\{B(e_i)\}_{i=1}^N$ is a frame for $\mathcal{H}$ and it is called an $\mathcal{L}_G(N,n)-$frame for $\mathcal{H}$\cite{deep1}. Let $F$ be a frame with Gramian matrix $\mathcal{G}.$ If $\Gamma$ is a graph with Laplacian matrix $\mathcal{L}$ such that $ \mathcal{G}= \mathcal{L},$ then $F$ is called a frame generated by the
 graph $\Gamma$ for $\mathcal{H}.$ In short, we call $F$ as a $\Gamma(N,n)-$frame for $\mathcal{H}.$

\begin{thm}\cite{bapat}\label{thm5point1}
	Let $\Gamma$ is a connected graph of $N$ vertices where $\eta$ and $\eta'$ denote the largest and smallest vertex degrees, respectively. If the eigenvalues of the Laplacian matrix $\mathcal{L}(\Gamma)$ are $\nu_1 \geq \nu_2 \geq \ldots \nu_N =0,$ then $\nu_1 \geq \eta + 1$ and $\nu_{N-1} \leq \frac{N}{N-1}\eta'.$
\end{thm}
The following theorem explores the relationship between a connected graph's Laplacian matrix and tight frames. Specifically, it considers a connected graph $\Gamma$ with $N$ vertices and examines a $\mathcal{L}_{\Gamma}(N,N-1)$ frame for $\mathbb{R}^{(N-1)}.$ The theorem establishes a condition under which this frame and its canonical dual pair, if tight,  is a two erasure spectrally optimal dual pair.

\begin{thm} \label{thm5point2}
	Let $\Gamma$ be a connected graph of $N$ vertices and $F= \{f_i\}_{i=1}^N$ is a $\mathcal{L}_{\Gamma}(N,N-1)$ frame for $\mathbb{R}^{(N-1)}.$ If $F$ is a  tight frame then $\left( F, S_{F}^{-1}F \right) \in 	\mathcal{E}_{2}^p .$
\end{thm}
\begin{proof}
	Let $F$ be a tight frame with bound $A.$ Then the corresponding frame operator is $S_F = A\,I_{(N-1) \times (N-1)},$ where $I_{(N-1) \times (N-1)}$  is  the ${(N-1) \times (N-1)}$ identity matrix. Since $\Gamma$ is connected, $\mathcal{L}(\Gamma)$ has $(N-1)$ non-zero eigenvalues. Recall that the Gramian matrix of $F$ is $\mathcal{G} = \mathcal{L}(\Gamma) = \mathcal{D}(\Gamma) - \mathcal{A}(\Gamma) .$ Since the non-zero eigenvalues of $\mathcal{G} $ are the same as the non-zero eigenvalues of the frame operator $S_F,$ then the eigenvalues of $\mathcal{G} $ are $A$ of multiplicity $(N-1) $ and $0$ with multiplicity $1.$ let $\eta$ and $\eta'$ are the largest and smallest vertex degree of $\Gamma.$ Then by Theorem\ref{thm5point1}, we have $A \geq \eta + 1$ and $A \leq \frac{N}{N-1}\eta'.$ Therefore, $\eta + 1 \leq \frac{N}{N-1}\eta'.$ Since $\Gamma$ is a simple graph, then $\eta \leq N-1.$ This in turn gives $\eta \leq \eta'.$ Thus  $\eta = \eta'.$ Note that $\frac{1}{A}\|f_i\|^2 = \langle f_i, S_{F}^{-1}f_i \rangle = \deg{g_i} .$ So, $\frac{1}{A}\|f_i\|^2  $ is a constant for all $1 \leq i \leq N.$   Using the fact that $N-1 = \sum\limits_{i=1}^N \frac{1}{A}\|f_i\|^2,$ we have $\frac{1}{A}\|f_i\|^2  = \frac{N-1}{N},\;1 \leq i \leq N.$ Then,
	$\mathcal{E}_{1}^p (F,  S_{F}^{-1}F) = \frac{1}{N} \left(\sum\limits_{i=1}^N \left|  \langle f_i, S_{F}^{-1}f_i \rangle \right|^p \right)^{\frac{1}{p}} = \frac{N-1}{N}.$ Hence, by Proposition \ref{prop3point1}, we have $\left( F, S_{F}^{-1}F \right) \in \mathcal{E}_{1}^p.$ Since $G$ is connected regular graph with $\deg(g_i) = N-1,$ then $G$ is complete and hence, $\mathcal{A}(\Gamma) = \begin{pmatrix}
		0 & 1  & \cdots &1 &1\\
		1 & 0  & \cdots &1 &1 \\
		\vdots &\vdots  &\vdots &\vdots &\vdots \\
		1 & 1  & \cdots &1 &  0 \\
		
	\end{pmatrix}.$  This leads to $\langle f_i, S_{F}^{-1}f_j \rangle \langle f_j, S_{F}^{-1}f_i \rangle$ is a constant for all $ i \neq j.$ Hence as in the proof of Theorem\ref{thm3point7}, $\left( F, S_{F}^{-1}F \right) \in \mathcal{E}_{2}^p.$
\end{proof}

The following theorem builds upon the results from Theorem \ref{thm5point2}, extending the analysis by incorporating the operator norm as a measure of error.

\begin{thm}
	Let $\Gamma$ be a connected graph of $N$ vertices and $F= \{f_i\}_{i=1}^N$ is a $\mathcal{L}_{\Gamma}(N,N-1)$ frame for $\mathbb{R}^{(N-1)}.$ If $F$ is a  tight frame then $\left( F, S_{F}^{-1}F \right) \in 	\mathcal{O}_{1}^p .$
\end{thm}

It may be noted, however, that the tightness condition of FF is not necessary for $(F, S_{F}^{-1}F)$ to be optimal dual pair under spectral radius, as illustrated by the example below.

\begin{example}
	
	Let $\mathcal{H} = \mathbb{C}^2.$	Consider a connected graph $\Gamma$  :\\
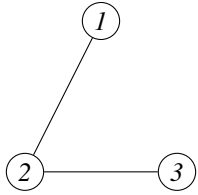
	$$\begin{tikzpicture}
		\node[draw, circle, inner sep=2pt] (A) at (0,0) {2};
		\node[draw, circle, inner sep=2pt] (B) at (2,0) {3};
		\node[draw, circle, inner sep=2pt] (C) at (1,2) {1};
		
		\draw (A) -- (B);
		\draw (A) -- (C);
		
	\end{tikzpicture}$$
	
	The Laplacian matrix of $\Gamma$ is $\mathcal{L}(\Gamma) =  \begin{pmatrix}
		1 & -1  & 0\\
		-1 & 2  & -1 \\
		0 & -1   &  1 \\
		
	\end{pmatrix}.$
	
	Take a frame $F= \{f_i\}_{i=1}^3 = \left\{ \begin{pmatrix}	\frac{1}{\sqrt{2}}  \\ \frac{1}{\sqrt{2}} \end{pmatrix}, \begin{pmatrix}	0  \\ \sqrt{2} \end{pmatrix}, \begin{pmatrix}	-\frac{1}{\sqrt{2}}  \\ \frac{1}{\sqrt{2}} \end{pmatrix} \right\}.$ Note that $F$ is not a tight frame. It is easy to see that the Gramian matrix of $F$ is $\mathcal{G}_F = \mathcal{L}(\Gamma).$ So $F$ is a $\mathcal{L}_\Gamma(3,2)$ frame for $\mathcal{H}.$ The canonical dual of $F$ is $S_{F}^{-1}F = \left\{ \begin{pmatrix}	\frac{1}{\sqrt{2}}  \\ \frac{1}{3\sqrt{2}} \end{pmatrix}, \begin{pmatrix}	0  \\ -\frac{\sqrt{2}}{3} \end{pmatrix}, \begin{pmatrix}	-\frac{1}{\sqrt{2}}  \\ \frac{1}{3\sqrt{2}} \end{pmatrix} \right\}.$ \\
	Now, $\langle f_i, S_{F}^{-1}f_i \rangle = \frac{2}{3},\;1 \leq i \leq 3.$ Therefore, $(F, S_{F}^{-1}F)$ is a $1-$uniform $(3,2)$ dual pair and hence by Proposition\ref{prop3point1}, $(F, S_{F}^{-1}F) \in \mathcal{E}_{1}^p.$ Additionally, $\langle f_i, S_{F}^{-1}f_j \rangle \langle f_j, S_{F}^{-1}f_i \rangle  = \frac{1}{9},\;\forall i \neq j.$ Therefore, $(F, S_{F}^{-1}F)$ is a $2-$uniform  dual pair and hence by Theorem\ref{thm3point7}, $(F, S_{F}^{-1}F) \in \mathcal{E}_{2}^p.$ However, $\|f_1\|\,\|S_{F}^{-1}f_1\| = \frac{\sqrt{5}}{3}, \|f_2\|\,\|S_{F}^{-1}f_2\| = \frac{2}{3}$ and $ \|f_3\|\,\|S_{F}^{-1}f_3\| = \frac{\sqrt{5}}{3}.  $ Thus by Theorem\ref{thm4point2}, $(F, S_{F}^{-1}F) \notin \mathcal{O}_{1}^p.$
\end{example}

The next theorem explores the properties of tight frames generated by a graph with no null vertices. It highlights the optimality of such frames and its canonical dual pair under spectral radius and operator norm. 

\begin{thm}
	Let $\Gamma$ be a graph of $N$ vertices with no null vertex and $F= \{f_i\}_{i=1}^N$ is a $\Gamma(N,n)$ frame for $\mathbb{C}^{(n)}.$ If $F$ is a  tight frame then $\left( F, S_{F}^{-1}F \right) $ is both $1-$erasure optimal dual pair under spectral radius and operator norm. 
\end{thm}
\begin{proof}
	Let $F$ be a tight frame with bound $A$ and $\Gamma_1, \Gamma_2,\ldots, \Gamma_{(N-n)}$ are the connected component of $\Gamma.$ Let the number of vertices in $\Gamma_i$ is $n_i.$ The Gramian matrix of $F$ is $$\mathcal{G} = \mathcal{L}(\Gamma) = \begin{pmatrix}
		\mathcal{L}(\Gamma_1) & 0  & \cdots &0 &0\\
		0 & \mathcal{L}(\Gamma_2)  & \cdots &0 &0 \\
		\vdots &\vdots  &\vdots &\vdots &\vdots \\
		0 & 0  & \cdots &0 &  \mathcal{L}(\Gamma_{(N-n)}) \\
		
	\end{pmatrix}.$$ 
As the frame operator of corresponding to $F$ is $S_F = AI_{n\times n},$ the eigenvalues of $\mathcal{L}(\Gamma)$ are $A$ with multiplicity $n$ and $0$ with multiplicity $(N-n).$ As each $\Gamma_i$ is connected, $\mathcal{L}(\Gamma_i)$ has rank $(n_i -1)$ and so $L(G_i) $ has eigenvalues $A$ with multiplicity $(n_i -1)$ and  $0$ with multiplicity 1. Now using the similar argument as in the proof of Theorem  \ref{thm5point2}, we can prove that each $\Gamma_i$ is regular. Let each $\Gamma_i$ is $r_i -$regular. Obviously $r_i > 0,\,$ for all $1 \leq i \leq N-n.$ Now, the adjacency matrix of $G$ is $$\mathcal{A}(\Gamma) =  \begin{pmatrix}
	\mathcal{A}(\Gamma_1) & 0  & \cdots &0 &0\\
	0 & \mathcal{A}(\Gamma_2)  & \cdots &0 &0 \\
	\vdots &\vdots  &\vdots &\vdots &\vdots \\
	0 & 0  & \cdots &0 &  \mathcal{A}(\Gamma_{(N-n)})\end{pmatrix} .$$ 
	Therefore, 
	$$ \mathcal{G} = \mathcal{L}(\Gamma) = \mathcal{D}(\Gamma) - \mathcal{A}(\Gamma) = \begin{pmatrix}
		r_{1}I - \mathcal{A}(\Gamma_1) & 0  & \cdots &0 &0\\
		0 & r_{2}I - \mathcal{A}(\Gamma_2)  & \cdots &0 &0 \\
		\vdots &\vdots  &\vdots &\vdots &\vdots \\
		0 & 0  & \cdots &0 &  r_{N-n}I - \mathcal{A}(\Gamma_{(N-n)})\end{pmatrix} .$$
	
	As $S_F = AI_{n\times n} = \Theta_{F}^* \Theta_F,$ we then have, $\mathcal{G}^2 = \Theta_F \Theta_{F}^{*} \Theta_F \Theta_{F}^{*} = \Theta_F S_F \Theta_{F}^* = A\mathcal{G}.$ This gives, $\left( r_{i} I -\mathcal{A}(\Gamma_i)    \right)^2 = A\left( r_{i} I -\mathcal{A}(\Gamma_i)    \right),\;1 \leq i \leq N-n.$ After simplifying we get 
	\begin{align}\label{eqn5point7}
		\left(\mathcal{A}(\Gamma_i)\right)^2 + (A- 2r_i)\mathcal{A}(\Gamma_i) + (r_{i}^2 - 2r_i)I) =0 , \;\;\forall i.
		\end{align} 
	 It is easy to see that each diagonal entries of each $\left(\mathcal{A}(\Gamma_i)\right)^2$ is $r_i$ and $\mathcal{A}(\Gamma_i)$ is $0$ and thus comparing the  diagonal entries on both side of \eqref{eqn5point7}, we get $r_{i} + r_{i}^2 - Ar_{i} =0,\; \forall i. $ Thus, $r_i = A -1,\; 1\leq i \leq N-n$ and hence $\Gamma$ is a $(A-1)-$regular graph. Hence, all the diagonal entries of $\mathcal{G}$ are same. This in turn gives $\|f_i\|^2$ is a constant for all $i.$ Using the fact that $n =\sum\limits_{i=1}^N \langle f_i,S_{F}^{-1}f_i \rangle = \sum\limits_{i=1}^N \frac{1}{A}\|f_i\|^2,$ we have $\|f_i\|^2 = \frac{An}{N},\; \forall i.$ Thus by \eqref{equation2}and \eqref{equation6}, we have $\mathcal{E}_{1}^p (F, S_{F}^{-1}F) = \frac{n}{N}= \mathcal{O}_{1}^p (F, S_{F}^{-1}F). $ Therefore, by Proposition\ref{prop3point1} and \ref{prop4point1} the result follows.
	
\end{proof}

Let $\Gamma$ be a graph with $n$ vertices. Let $F^{1} = \{f^{1}_i\}_{i=1}^N$ and  $F^{2} = \{f^{2}_i\}_{i=1}^N$ be two $\Gamma(N,n)$ frame for $\mathcal{H}.$  Using theorem 3.6 in \cite{deep1}, it can be easily calculate that $S_{F^{2}} = U S_{F^{2}} U^{*},$ where $F^{2} = U F^{1},$ for some unitary operator $U.$ By looking at Theorem \ref{thm3point7} and \ref{thm4point4unitary}, we can conclude the following proposition.

\begin{prop}
	 $F^{1} = \{f^{1}_i\}_{i=1}^N$ and  $F^{2} = \{f^{2}_i\}_{i=1}^N$ be two $\Gamma(N,n)$ frame for $\mathcal{H}.$ Then the following holds:
	 	\begin{enumerate}
	 	\item [{\em (i)}] $\left(F^1 , S_{F^1}^{-1}F^{1}  \right) \in \mathcal{E}_{i}^p $ if and only if $\left(F^2 , S_{F^2}^{-1}F^{2}  \right) \in \mathcal{E}_{i}^p,\;$ for $i=1,2.$
	 	
	 	\item [{\em (ii)}]  $\left(F^1 , S_{F^1}^{-1}F^{1}  \right) \in \mathcal{O}_{1}^p $ if and only if $\left(F^2 , S_{F^2}^{-1}F^{2}  \right) \in \mathcal{O}_{1}^p.$
	 \end{enumerate}
\end{prop}

The following lemma addresses the structure of dual frames for a specific class of frames associated with a connected graph. It provides a detailed description of how the dual frames can be characterized, particularly when $N>n.$ If $N=n,$ a frame has only its canonical dual. This result is important for understanding the flexibility and structure of dual frames in the context of graph-generated frames.

\begin{lem}
	Suppose $\Gamma$ is a connected graph with $N$ vertices and $F= \{f_i\}_{i=1}^N$ be a $\Gamma(N, N-1)$ frame for $\mathbb{C}^{N-1}$ with frame operator $S_F.$ Then the family of dual frames of $F$ is of the form $\left\{\left\{ S_{F}^{-1}f_i + \frac{c_i}{c_1}h  \right\}_{i=1}^N : h \in \mathbb{C}^{N-1}\right\},$ when $ \{f_i\}_{i=1}^N$ satisfy $\sum\limits_{i=1}^N c_i f_i = 0,\;$ with $c_1 \neq 0.$
\end{lem}

\begin{proof}
	We have $\Theta_F^{*}e_i = f_i,\,1 \leq i \leq N$ and the Gramian matrix $\mathcal{G} = \mathcal{L}(\Gamma) = \Theta_F \Theta_F^{*},$ where $\mathcal{L}(\Gamma)$ is the Laplacian matrix of the graph $\Gamma.$ It is well-known that any dual of $F$ is of the form $\left\{ S_{F}^{-1}f_i + h_i  \right\}_{i=1}^N, $ where $\sum\limits_{i=1}^N \langle f, f_i \rangle h_i = 0.$   Thus we have $\Theta_F^{*}\tilde{f} = 0,$ for any $f \in \mathcal{H},$ where $\tilde{f} = [\langle f,h_1 \rangle, \cdots , \langle f,h_n \rangle ]^{t}$ and $[\langle f,h_1 \rangle, \cdots , \langle f,h_n \rangle ]^{t}$ denotes the transpose of $[\langle f,h_1 \rangle, \cdots , \langle f,h_n \rangle ].$ This gives $\mathcal{L}(\Gamma) \tilde{f} = 0.$ Since $\Gamma$ is a connected graph, thus rank$\left(\mathcal{L}(\Gamma)\right) = N-1 $ and hence $\dim\left(N(\Theta_F) \right) = 1,$ where $N(\Theta_F)$ denote the null space of $\Theta_F.$ Note that $L\left( [c_1,c_2, \cdots, c_n]^{t}  \right) = 0.$ Therefore, $N(\Theta_F) = \left\{\alpha [c_1, c_2, \cdots, c_n]^{t} : \alpha \in \mathbb{C}\right\}.$ Thus, $\tilde{f} = [\langle f,h_1 \rangle, \cdots , \langle f,h_n \rangle ]^{t} = [\alpha'c_1, \alpha'c_2, \cdots, \alpha'c_n]^{t},$ for some $\alpha' \in \mathbb{C}.$  This gives, $\langle f,h_i \rangle = \alpha'c_i,\,1 \leq i \leq N.$ So, $\frac{\langle f,h_i \rangle}{\langle f,h_1 \rangle} = \frac{c_i}{c_1},$ for all $i$ and  for all $f \in \mathcal{H}.$ This in turn leads to $\langle f, \bar{c_1}h_i - \bar{c_i}h_1 \rangle = 0,\;1 \leq i \leq N, \forall f \in \mathcal{H}.$ This gives $\bar{c_1}h_i - \bar{c_i}h_1 = 0$ and so $h_i = \frac{\bar{c_i}}{\bar{c_1}}h_1,\;1 \leq i \leq N.$ Therefore, any dual is of the form  $\left\{\left\{ S_{F}^{-1}f_i + \frac{c_i}{c_1}h  \right\}_{i=1}^N : h \in \mathbb{C}^{N-1}\right\}.$
\end{proof}

A frame $F= \{f_i\}_{i=1}^N$ is called $k-$independent if any $k$ vectors are linearly independent \cite{sali}. The following lemma provides the expression for $\mathcal{E}_{1}^p\left( F, S_{F}^{-1}F \right),$ when $F$ is $(N-1)-$independent.
\begin{lem}\label{lemma5point6}
	let $\Gamma$ be a connected graph with $N$ vertices. If $F= \{f_i\}_{i=1}^N$ be an $(N-1)-$independent $\mathcal{L}_{\Gamma}(N,N-1)$ frame for $\mathbb{C}^{N-1}.$ Then,   $\mathcal{E}_{1}^p\left( F, S_{F}^{-1}F \right) = \left\{\dfrac{1}{N} \sum\limits_{i=1}^N \left(1- \dfrac{c_{i}^2}{c_{1}^2 + c_{2}^2 + \cdots + c_{N}^2}\right)^p\right\}^{\frac{1}{p}} .$  
\end{lem}

\begin{proof}
	As  $ \{f_i\}_{i=1}^N$ is not a basis, then  $ \{f_i\}_{i=1}^N$ satisfy the relation $\sum\limits_{i=1}^N c_i f_i =0,\,c_i \in \mathbb{C}$ and for some $c_i \neq 0.$ We claim that $c_i \neq 0,\;\forall i.$ Suppose not. Then for some $j,\;c_j =0.$ This gives, $\sum\limits_{i \neq j} c_i f_i =0.$ Using the face that $F$ is  $(N-1)-$independent we have $c_i =0, 1\leq i \leq N,$ which arise a contradiction.\\ 
	Let $\mathcal{L}(\Gamma)$ be the Laplacian matrix of $\Gamma$ with eigenvalues $\lambda_1, \lambda_2, \cdots, \lambda_{N-1}, 0.$ Since, $\Gamma$ is a connected graph then $\dim{(\mathcal{L}(\Gamma))} = 1.$ It is easy to see that $\mathcal{L}(\Gamma) \begin{pmatrix} c_1\\c_2\\ \vdots \\c_N \end{pmatrix} = \Theta_F \Theta_F^{*} \begin{pmatrix} c_1\\c_2\\ \vdots \\c_N \end{pmatrix}  = \Theta_F(\sum\limits_{i=1}^N c_i f_i ) = 0.$ Hence, $ \begin{pmatrix} c_1\\c_2\\ \vdots \\c_N \end{pmatrix} $ is an eigenvector corresponding to the eigenvalue $0.$ Now $\mathcal{L}(\Gamma)$ can be written as $\mathcal{L}(\Gamma) = MDM^*,$ where $M$ is the matrix containing the eigenvectors column wise and  $D = \text{diag}(\lambda_1,\lambda_2, \ldots, \lambda_{N-1},0).$ Hence, the last column of $M$ is  $ \begin{pmatrix} \frac{c_1}{\sqrt{c_{1}^2 + c_{2}^2 + \cdots + c_{N}^2}}\\\frac{c_2}{\sqrt{c_{1}^2 + c_{2}^2 + \cdots + c_{N}^2}}\\ \vdots \\\frac{c_N}{\sqrt{c_{1}^2 + c_{2}^2 + \cdots + c_{N}^2}} \end{pmatrix} .$ By Lemma3.9 in \cite{deep1},  $S_F = \text{diag}(\lambda_1, \lambda_2, \cdots, \lambda_{N-1}).$ So, $S_{F}^{\frac{1}{2}} = \text{diag}(\sqrt{\lambda_1}, \sqrt{\lambda_2}, \cdots, \sqrt{\lambda_{N-1}})= D_1$ (say). Taking $\xi_{j}$ to be the $j^{th}$ column of $M$ and $\eta = \begin{pmatrix} \frac{c_1}{\sqrt{c_{1}^2 + c_{2}^2 + \cdots + c_{N}^2}}\\\frac{c_2}{\sqrt{c_{1}^2 + c_{2}^2 + \cdots + c_{N}^2}}\\ \vdots \\\frac{c_N}{\sqrt{c_{1}^2 + c_{2}^2 + \cdots + c_{N}^2}} \end{pmatrix}, $  we have
	\begin{align*}
		\langle  S_{F}^{-1}f_i, f_i \rangle = \left\|  S_{F}^{-\frac{1}{2}}f_i    \right\|^2 = \left\| D_{1}^{-1} D_{1} M_{1}^{*}(e_i)  \right\|^2 = \left\| M_{1}^{*}(e_i)  \right\|^2 &= \sum\limits_{j=1}^{N-1}|\langle \xi_j, e_i \rangle |^2 + |\langle \eta, e_i \rangle |^2 -  |\langle \eta, e_i \rangle |^2 \\&= \|e_i\|^2 -  |\langle \eta, e_i \rangle |^2  = 1- \frac{c_{i}^2}{c_{1}^2 + c_{2}^2 + \cdots + c_{N}^2}, 
	\end{align*}
	for all $ 1 \leq i \leq N$ and 	$M_1$ is a sub-matrix of $M$ formed by the first $N-1$ columns. Therefore by \eqref{equation2},  $\mathcal{E}_{1}^p\left( F, S_{F}^{-1}F \right) = \left\{\dfrac{1}{N} \sum\limits_{i=1}^N \left(1- \dfrac{c_{i}^2}{c_{1}^2 + c_{2}^2 + \cdots + c_{N}^2}\right)^p\right\}^{\frac{1}{p}} .$ \\
	 
\end{proof}

The following theorem delves into the properties of frames generated by the Laplacian matrix of a connected graph. It specifically addresses the uniqueness of dual frames under certain conditions, emphasizing the exclusivity of a particular dual frame pair when the frame vectors sum to zero. This result is crucial for understanding the uniqueness and optimality of such frames.

\begin{thm}
	let $\Gamma$ be a connected graph with $N$ vertices. If $F= \{f_i\}_{i=1}^N$ be an $\mathcal{L}_{\Gamma}(N,N-1)$ frame for $\mathbb{C}^{N-1}$ such that $\sum\limits_{i=1}^N f_i = 0.$ Then,   $\left( F, S_{F}^{-1}F \right) \in \mathcal{E}_{1}^p.$ Further, $F$ has no other dual $G'$ such that $(F,G')  \in \mathcal{E}_{1}^p.$
\end{thm}

\begin{proof}
	Let $\mathcal{L}(\Gamma)$ be the Laplacian matrix of $\Gamma$ with eigenvalues $\lambda_1, \lambda_2, \cdots, \lambda_{N-1}, 0.$ Since, $\Gamma$ is a connected graph then $\dim{(\mathcal{L}(\Gamma))} = 1.$ Taking $c_i =1$ as in Lemma\ref{lemma5point6}, we have $\mathcal{E}_{1}^p\left( F, S_{F}^{-1}F \right) = \frac{N-1}{N} $ and hence by Proposition\ref{prop3point1}, $\left( F, S_{F}^{-1}F \right) \in \mathcal{E}_{1}^p.$ \\
	Suppose $\tilde{G} = \{g_i\}_{i=1}^N =  \{S_{F}^{-1}f_i + h\}_{i=1}^N $ be a dual of $F$ such that $\mathcal{E}_{1}^p(F, \tilde{G}) = \frac{N-1}{N}.$ Then by Proposition\ref{prop3point1}, we have $\langle f_i, g_i \rangle = \frac{N-1}{N},\;1 \leq i \leq N.$ This leads to, $\left\|  S_{F}^{-\frac{1}{2}}f_i    \right\|^2 + \langle f_i, h \rangle = \frac{N-1}{N},\;1\leq i \leq N.$ Thus, $\langle f_i, h \rangle = 0,\,\forall i$ and therefore, $h= 0.$ Hence the result follows. 
 \end{proof}

\section{Acknowledgment}
	 The authors are grateful to the Mohapatra Family Foundation and the College of Graduate Studies of the University of Central Florida for their support during this research.

\bibliographystyle{amsplain}

\end{document}